\documentclass[10pt]{amsart}
\usepackage{a4,amssymb,times,amsmath,cancel,mathrsfs,multirow,textcomp}
\usepackage{graphicx,color,epsfig}

\usepackage[all]{xy}
\usepackage{tikz}
\usetikzlibrary{calc,3d,arrows}
\def\r{\rightarrow}
\def\l{\leftarrow}
\def\ot{\leftarrow}
\def\u{\uparrow}
\def\d{\downarrow}
\def\Id{\id}
\def\opp{\mathrm{opp}}
\def\triv{o}

\def\gen{{\gamma}}
\let\cal\mathcal
\def\cA{{\mathtt A}}
\def\cB{{\mathtt B}}
\def\cE{{\cal E}}
\def\cF{{\cal F}}
\def\cG{{\cal G}}
\def\cH{{\cal H}}
\def\cI{{\cal I}}
\def\cJ{{\cal J}}
\def\cK{{\cal K}}
\def\cL{{\cal L}}
\def\cM{{\cal M}}
\def\cN{{\cal N}}
\def\cO{{\cal O}}
\def\cP{{\cal P}}
\def\cQ{{\cal Q}}
\def\cR{{\cal R}}
\def\cS{{\cal S}}
\def\cT{{\cal T}}
\def\cU{{\cal U}}
\def\cV{{\cal V}}
\def\cW{{\cal W}}
\def\cX{{\cal X}}
\def\cY{{\cal Y}}
\def\cZ{{\cal Z}}
\def\pZ{{\cal Z}}
\def\eps{{\epsilon}}

\let\blb\mathbb
\def\normalize{\mathtt {normalize}}
\def\CM{{\blb {CM}}}
\def\H{{\blb M}}
\def\HM{{\blb {MM}}}
\def\NCCR{{\blb {NCCR}}}
\def\DD{{\blb D}}
\def\UU{{\blb U}}
\def\XX{{\blb X}}
\def\FF{{\blb F}}
\def\cFF{{\blb F}}
\def\cGG{{\blb G}}
\def\QQ{{\blb Q}}
\def\GG{{\blb G}}
\def \PP{{\blb P}}
\def \AA{{\blb A}}
\def \ZZ{{\blb Z}}
\def \T{{\blb T}}
\def \TT{{\blb T}}
\def \NN{{\blb N}}
\def \Rl{{\blb R}}
\def \HH{{\blb H}}
\def \VV{{\blb V}}
\def \SS{{\blb S}}

\newcommand{\se}[1]{\begin{equation*}\begin{split}#1\end{split}\end{equation*}}
\newcommand{\ceil}[1]{\lceil #1 \rceil}
\newcommand{\cone}[1]{[ #1 ]}
\newcommand{\lvect}[1]{\overleftarrow{#1}}
\newcommand{\wis}[1]{{\text{\em \usefont{OT1}{cmtt}{m}{n} #1}}}
\newcommand{\del}[1]{{#1}^{-1}}
\newcommand{\C}{\mathbb{C}}
\newcommand{\N}{\mathbb{N}}
\newcommand{\Z}{\mathbb{Z}}
\newcommand{\R}{\mathbb{R}}
\newcommand{\RZ}{\zeta}
\newcommand{\cMF}{\mathscr{Mf}}
\newcommand{\FM}{\mathscr{F}}
\newcommand{\rest}{\mathrm{rest}}
\def\cLL{{\mathscr L}}
\def\ccE{{\mathscr E}}
\def\ccM{{\mathscr M}}
\def\ccU{{\mathscr U}}
\def\cKK{{\mathscr K}}
\def\ccF{{\mathscr F}}
\def\ccG{{\mathscr G}}
\def\cTT{{\mathscr T}}

\newcommand{\Rplus}{\mathbb{R_+}}
\newcommand{\genmu}{\text{\textmu}}
\newcommand{\atimes}{\underset{\text{{\it\tiny A--A}}}{\otimes}}
\newcommand{\stimes}{\underset{\text{{\it\tiny S--S}}}{\otimes}}
\newcommand{\unit}{\mathsf{1}}
\newcommand{\Affine}{\mathbb{A}}
\newcommand{\wt}{*+{\circ}}
\newcommand{\bk}{*+{\bullet}}
\newcommand{\grp}[1]{\mathsf{#1}}
\newcommand{\lie}[1]{\mathsf{#1}}
\newcommand{\mut}{\mathsf{mut}}
\newcommand{\vtx}[1]{*+[o][F-]{\scriptscriptstyle #1}}
\newcommand{\sqvtx}[1]{*+[F]{\scriptscriptstyle #1}}
\newcommand{\block}[1]{*+[F]{#1}}
\newcommand{\mtx}[2]{\left (\begin{array}{#1} #2 \end{array}\right )}
\newcommand{\St}{\textrm{St}}
\newcommand{\Tr}{\textrm{Tr}}
\newcommand{\rank}{\textrm{rank}}
\newcommand{\coker}{\textrm{coker}}
\newcommand{\Cok}{\mathtt{Cok}}
\newcommand{\cycle}{\circlearrowright}
\newcommand{\cycl}{\cycle}
\newcommand{\proj}{\wis{proj}}
\newcommand{\GT}{\ensuremath{\mathsf{G}}}
\newcommand{\TG}{\GT}
\newcommand{\grep}{\ensuremath{\mathsf{grep}}}
\newcommand{\Spec}{\ensuremath{\mathsf{Spec}}}
\newcommand{\giss}{\ensuremath{\mathsf{giss}}}
\newcommand{\gtiss}{\ensuremath{\mathsf{gtiss}}}
\newcommand{\gtrep}{\ensuremath{\mathsf{gtrep}}}
\newcommand{\siss}{\ensuremath{\mathsf{siss}}}
\newcommand{\srep}{\ensuremath{\mathsf{srep}}}
\newcommand{\<}{\langle}
\renewcommand{\>}{\rangle}
\newcommand{\spl}{\mathsf{split}}
\newcommand{\Aut}{\ensuremath{\mathsf{Aut}}}
\newcommand{\lift}{\ensuremath{\mathsf{lift}}}
\newcommand{\Tor}{\ensuremath{\mathsf{Tor}}}
\newcommand{\Mor}{\ensuremath{\mathsf{Mor}}}
\newcommand{\Bimod}{\ensuremath{\mathsf{Bimod\,}}}
\newcommand{\aut}{\Aut}
\newcommand{\degree}{\text{deg}}
\newcommand{\lcm}{\text{lcm}}
\newcommand{\rdeg}{\text{rdeg}}
\newcommand{\vdeg}{\text{vdeg}}
\newcommand{\coeff}{\text{cf}}
\newcommand{\ind}{{\xy 0;/r.08pc/:
(3,-3); (5,-3) **@{-}; (5,5) **@{-}; (0,2) **@{-};
(3,-3); (3,4) **@{-}
\endxy}}
\newcommand{\bigt}{{\LARGE \mbox{\textsf{T}}}}
\newcommand{\tiss}{\ensuremath{\mathsf{tiss}}}
\newcommand{\simples}{\ensuremath{\mathsf{simples}}}
\newcommand{\End}{\ensuremath{\mathsf{End}}}
\newcommand{\Stab}{\ensuremath{\mathsf{Stab}}}
\newcommand{\Trans}{\mathsf{Trans}}
\newcommand{\Ext}{\mathsf{Ext}}
\newcommand{\Pic}{\mathsf{Pic}}
\newcommand{\Eqv}{\mathsf{Eqv}}
\newcommand{\Inv}{\mathsf{Inv}}
\newcommand{\Dim}{\mathsf{Dim \,}}
\newcommand{\sign}{\mathsf{sign \,}}
\newcommand{\image}{\mathsf{Im \,}}
\newcommand{\vierkant}[1]{\boxed{\scriptscriptstyle{#1}}}
\newcommand{\CCH}{\ensuremath{\mathsf{CH}}}
\newcommand{\spec}{\wis{spec}}
\font\ef = eufb10
\newcommand{\ideal}[1]{\ensuremath{\text{\ef{#1}}}}
\newcommand{\Symm}{\ideal{S}}
\newcommand{\sto}{\Rightarrow}
\newcommand{\Span}{\mathrm{Span}}

\newtheorem{lemma}{Lemma}[section]
\newtheorem{proposition}[lemma]{Proposition}
\newtheorem{theorem}[lemma]{Theorem}

\theoremstyle{definition}

\newtheorem{example}[lemma]{Example}

\newtheorem{question}[lemma]{Question}

{

}

\theoremstyle{remark}

\newtheorem{remark}[lemma]{Remark}

\def\thenotation{}

\newcommand{\prf}{\noindent\textbf{Proof.}~}
\newcommand{\eop}{\hfill $\square$\\}
\newcommand{\csum}[1]{\stackrel{\scriptscriptstyle\#}{\scriptscriptstyle#1}}
\newcommand{\rep}{\ensuremath{\mathsf{rep}}}
\newcommand{\brep}{\ensuremath{\mathsf{brep}}}
\newcommand{\trep}{\ensuremath{\mathsf{trep}}}
\newcommand{\Trep}{\ensuremath{\mathsf{trep}}}
\newcommand{\iss}{\ensuremath{\mathsf{iss}}}
\newcommand{\Sup}{\ensuremath{\mathsf{Sup}}}
\newcommand{\coh}{\ensuremath{\mathsf{Coh}}}
\newcommand{\Proj}{\ensuremath{\mathsf{Proj}}}
\newcommand{\moss}{\mathsf{moss}}
\newcommand{\Mat}{\mathsf{Mat}}
\newcommand{\ress}{\mathsf{ress}}
\newcommand{\Null}{\mathsf{Null}}
\newcommand{\Nil}{\mathsf{Nil}}
\newcommand{\Rep}{\mathsf{Rep}}
\newcommand{\RRep}{\mathsf{RRep}}
\newcommand{\Hom}{\mathtt{Hom}}
\newcommand{\RHom}{\textrm{RHom}}
\newcommand{\Ker}{\textrm{Ker}}
\newcommand{\Image}{\textrm{Im}}
\newcommand{\GL}{\ensuremath{\mathsf{GL}}}
\newcommand{\id}{\mathbf{1}}
\newcommand{\defect}{\textrm{def}}
\newcommand{\diag}{\wis{diag}}
\newcommand{\quot}{/\!\!/}
\newcommand{\A}{\ensuremath{\mathbb{A}}}
\newcommand{\V}{\ensuremath{\mathbb{V}}}
\newcommand{\lt}{\ensuremath{\mathsf{lt}}}

\newcommand{\e}{\varepsilon}
\renewcommand{\t}[1]{\textnormal{#1}}

\newcommand{\tcat}{\cT}
\newcommand{\carr}{\mathtt R}
\newcommand{\ccyc}{\mathtt E}
\newcommand{\PM}{\cP}
\renewcommand{\Rplus}{\R_+}
\newcommand{\Ob}{\mathtt{Ob}\,}

\renewcommand{\H}{\mathtt{H}\,}
\newcommand{\Tw}{\mathtt{Tw}\,}
\newcommand{\cHH}{\mathtt{H}}
\newcommand{\Hoch}{\mathtt{HochC}}
\newcommand{\Coh}{\mathtt{Coh}\,}
\newcommand{\Perf}{\mathtt{Perf}\,}
\newcommand{\Sing}{\mathtt{Sing}\,}
\newcommand{\cJA}{\mathtt{Jac}\,}
\newcommand{\cQA}{\mathtt{Gtl}\,}
\def\cC{{\mathtt C}}

\def\Fuk{{\mathtt {Fuk}}}
\def\WFuk{{\mathtt {WFuk}}}
\def\MF{{\mathtt {MF}}}
\def\TF{{\mathtt {F}}}
\def\cD{{\mathtt D}}
\def\D{{\mathtt D}}
\def\cDb{{\mathtt D^b}}
\def\fuk{{\mathtt {fuk}}}
\def\mf{{\mathtt {mf}}}
\def\Pg{{\mathtt {P}}}

\def\qpol{\mathrm{Q}}
\newcommand{\mirror}[1]{\mathrel{\reflectbox{$#1$}}}
\newcommand{\mmirror}[1]{\!\!\mathrel{\reflectbox{$#1$}}\!\!}
\newcommand{\rect}{\mathrm{R}}
\newcommand{\smatrix}[1]{\left(\begin{smallmatrix}#1\end{smallmatrix}\right)}
\newcommand{\Mod}{\ensuremath{\mathtt{Mod}}}

\title{Toric systems and mirror symmetry}

\author{Raf Bocklandt}
\address{Raf Bocklandt\\
School of Mathematics and Statistics\\
Herschel Building\\
Newcastle University\\
Newcastle upon Tyne\\
NE1 7RU\\
UK}
\email{raf.bocklandt@gmail.com}

\xyoption{all}

\begin{document}
\begin{abstract}
In \cite{perlinghille} Hille and Perling associate to every cyclic full strongly exceptional sequence of
line bundles on a toric weak Fano surface a toric system, which defines a new toric surface.
In this note we interprete this construction as an instance of mirror symmetry and extend it
to a duality on the set toric weak Fano surfaces equiped with a cyclic full strongly exceptional sequence.
\end{abstract}

\maketitle

\section{Reflexive polygons and weak Fano surfaces}
A convex integral polygon in $\Z^2$ that has exactly one internal lattice point is called a \emph{reflexive polygon}.
Up to integral affine transformations, there are precisely $16$ reflexive polygons, which are shown in the table below:
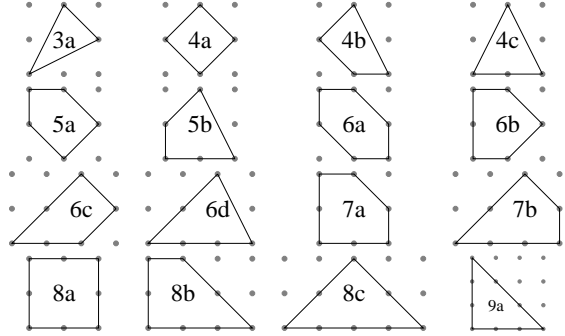
\begin{center}
\begin{tabular}{cccc}
\resizebox{!}{1cm}{\begin{tikzpicture}
\filldraw [gray] 
(-.5,-.5) circle (2/2pt) (0,-.5) circle (2/2pt) (.5,-.5) circle (2/2pt)
(-.5,0) circle (2/2pt)  (.5,0) circle (2/2pt)
(-.5,.5) circle (2/2pt) (0,.5) circle (2/2pt) (.5,.5) circle (2/2pt);
\draw (0,0) node {3a};
\draw (-.5,-.5) -- (.5,0) -- (0,.5) -- (-.5,-.5);  
\end{tikzpicture}}&
\resizebox{!}{1cm}{\begin{tikzpicture}
\filldraw [gray] 
(-.5,-.5) circle (2/2pt) (0,-.5) circle (2/2pt) (.5,-.5) circle (2/2pt)
(-.5,0) circle (2/2pt)  (.5,0) circle (2/2pt)
(-.5,.5) circle (2/2pt) (0,.5) circle (2/2pt) (.5,.5) circle (2/2pt);
\draw (0,-.5) -- (.5,0) -- (0,.5) -- (-.5,0) -- (0,-.5);  
\draw (0,0) node {4a};
\end{tikzpicture}}&\resizebox{!}{1cm}{\begin{tikzpicture}
\filldraw [gray] 
(-.5,-.5) circle (2/2pt) (0,-.5) circle (2/2pt) (.5,-.5) circle (2/2pt)
(-.5,0) circle (2/2pt)  (.5,0) circle (2/2pt)
(-.5,.5) circle (2/2pt) (0,.5) circle (2/2pt) (.5,.5) circle (2/2pt);
\draw (.5,-.5) -- (0,.5) -- (-.5,0) -- (0,-.5) -- (.5,-.5);  
\draw (0,0) node {4b};
\end{tikzpicture}}&\resizebox{!}{1cm}{\begin{tikzpicture}
\filldraw [gray] 
(-.5,-.5) circle (2/2pt) (0,-.5) circle (2/2pt) (.5,-.5) circle (2/2pt)
(-.5,0) circle (2/2pt)  (.5,0) circle (2/2pt)
(-.5,.5) circle (2/2pt) (0,.5) circle (2/2pt) (.5,.5) circle (2/2pt);
\draw (.5,-.5) -- (0,.5) -- (-.5,-.5) -- (.5,-.5);  
\draw (0,0) node {4c};
\end{tikzpicture}}\\
\resizebox{!}{1cm}{\begin{tikzpicture}
\filldraw [gray] 
(-.5,-.5) circle (2/2pt) (0,-.5) circle (2/2pt) (.5,-.5) circle (2/2pt)
(-.5,0) circle (2/2pt)  (.5,0) circle (2/2pt)
(-.5,.5) circle (2/2pt) (0,.5) circle (2/2pt) (.5,.5) circle (2/2pt);
\draw (.5,0) -- (0,.5) -- (-.5,.5) -- (-.5,0) -- (0,-.5) -- (.5,0);  
\draw (0,0) node {5a};
\end{tikzpicture}}&\resizebox{!}{1cm}{\begin{tikzpicture}
\filldraw [gray] 
(-.5,-.5) circle (2/2pt) (0,-.5) circle (2/2pt) (.5,-.5) circle (2/2pt)
(-.5,0) circle (2/2pt)  (.5,0) circle (2/2pt)
(-.5,.5) circle (2/2pt) (0,.5) circle (2/2pt) (.5,.5) circle (2/2pt);
\draw (.5,-.5) -- (0,.5) -- (-.5,0) -- (-.5,-.5) -- (.5,-.5);  
\draw (0,0) node {5b};
\end{tikzpicture}}&\resizebox{!}{1cm}{\begin{tikzpicture}
\filldraw [gray] 
(-.5,-.5) circle (2/2pt) (0,-.5) circle (2/2pt) (.5,-.5) circle (2/2pt)
(-.5,0) circle (2/2pt)  (.5,0) circle (2/2pt)
(-.5,.5) circle (2/2pt) (0,.5) circle (2/2pt) (.5,.5) circle (2/2pt);
\draw (.5,0) -- (0,.5) -- (-.5,.5) -- (-.5,0) -- (0,-.5) -- (.5,-.5) -- (.5,0);  
\draw (0,0) node {6a};
\end{tikzpicture}}&\resizebox{!}{1cm}{\begin{tikzpicture}
\filldraw [gray] 
(-.5,-.5) circle (2/2pt) (0,-.5) circle (2/2pt) (.5,-.5) circle (2/2pt)
(-.5,0) circle (2/2pt)  (.5,0) circle (2/2pt)
(-.5,.5) circle (2/2pt) (0,.5) circle (2/2pt) (.5,.5) circle (2/2pt);
\draw (.5,0) -- (0,.5) -- (-.5,.5) -- (-.5,-.5) -- (0,-.5) -- (.5,0);  
\draw (0,0) node {6b};
\end{tikzpicture}}\\
\resizebox{!}{1cm}{\begin{tikzpicture}
\filldraw [gray] 
(-.5,-.5) circle (2/2pt) (0,-.5) circle (2/2pt) (.5,-.5) circle (2/2pt)
(-.5,0) circle (2/2pt)  (.5,0) circle (2/2pt)
(-.5,.5) circle (2/2pt) (0,.5) circle (2/2pt) (.5,.5) circle (2/2pt)
(-1,-.5) circle (2/2pt) (-1,0) circle (2/2pt) (-1,.5) circle (2/2pt);
\draw (.5,0) -- (0,.5) -- (-1,-.5) -- (0,-.5) -- (.5,0);  
\draw (0,0) node {6c};
\end{tikzpicture}}&
\resizebox{!}{1cm}{\begin{tikzpicture}
\filldraw [gray] 
(-.5,-.5) circle (2/2pt) (0,-.5) circle (2/2pt) (.5,-.5) circle (2/2pt)
(-.5,0) circle (2/2pt)  (.5,0) circle (2/2pt)
(-.5,.5) circle (2/2pt) (0,.5) circle (2/2pt) (.5,.5) circle (2/2pt)
(-1,-.5) circle (2/2pt) (-1,0) circle (2/2pt) (-1,.5) circle (2/2pt);
\draw (0,.5) -- (-1,-.5) -- (.5,-.5) -- (0,.5);  
\draw (0,0) node {6d};
\end{tikzpicture}}&\resizebox{!}{1cm}{\begin{tikzpicture}
\filldraw [gray] 
(-.5,-.5) circle (2/2pt) (0,-.5) circle (2/2pt) (.5,-.5) circle (2/2pt)
(-.5,0) circle (2/2pt)  (.5,0) circle (2/2pt)
(-.5,.5) circle (2/2pt) (0,.5) circle (2/2pt) (.5,.5) circle (2/2pt);
\draw (.5,0) -- (0,.5) -- (-.5,.5) -- (-.5,-.5) -- (.5,-.5) -- (.5,0); 
\draw (0,0) node {7a};
\end{tikzpicture}}&
\resizebox{!}{1cm}{\begin{tikzpicture}
\filldraw [gray] 
(-.5,-.5) circle (2/2pt) (0,-.5) circle (2/2pt) (.5,-.5) circle (2/2pt)
(-.5,0) circle (2/2pt)  (.5,0) circle (2/2pt)
(-.5,.5) circle (2/2pt) (0,.5) circle (2/2pt) (.5,.5) circle (2/2pt)
(-1,-.5) circle (2/2pt) (-1,0) circle (2/2pt) (-1,.5) circle (2/2pt);
\draw (.5,0) -- (0,.5) -- (-1,-.5) -- (.5,-.5) --(.5,0); 
\draw (0,0) node {7b};
\end{tikzpicture}}\\
\resizebox{!}{1cm}{\begin{tikzpicture}
\filldraw [gray] 
(-.5,-.5) circle (2/2pt) (0,-.5) circle (2/2pt) (.5,-.5) circle (2/2pt)
(-.5,0) circle (2/2pt)  (.5,0) circle (2/2pt)
(-.5,.5) circle (2/2pt) (0,.5) circle (2/2pt) (.5,.5) circle (2/2pt);
\draw (.5,.5) -- (-.5,.5) -- (-.5,-.5) -- (.5,-.5) -- (.5,.5);  
\draw (0,0) node {8a};
\end{tikzpicture}}&\resizebox{!}{1cm}{\begin{tikzpicture}
\filldraw [gray] 
(-.5,-.5) circle (2/2pt) (0,-.5) circle (2/2pt) (.5,-.5) circle (2/2pt)
(-.5,0) circle (2/2pt)  (.5,0) circle (2/2pt)
(-.5,.5) circle (2/2pt) (0,.5) circle (2/2pt) (.5,.5) circle (2/2pt)
(1,-.5) circle (2/2pt) (1,0) circle (2/2pt) (1,.5) circle (2/2pt);
\draw (0,.5) -- (-.5,.5) -- (-.5,-.5) -- (1,-.5) -- (0,.5);  
\draw (0,0) node {8b};
\end{tikzpicture}}&\resizebox{!}{1cm}{\begin{tikzpicture}
\filldraw [gray] 
(-.5,-.5) circle (2/2pt) (0,-.5) circle (2/2pt) (.5,-.5) circle (2/2pt)
(-.5,0) circle (2/2pt)  (.5,0) circle (2/2pt)
(-.5,.5) circle (2/2pt) (0,.5) circle (2/2pt) (.5,.5) circle (2/2pt)
(1,-.5) circle (2/2pt) (1,0) circle (2/2pt) (1,.5) circle (2/2pt)
(-1,-.5) circle (2/2pt) (-1,0) circle (2/2pt) (-1,.5) circle (2/2pt);
\draw (0,.5) -- (1,-.5) -- (-1,-.5) -- (0,.5);  
\draw (0,0) node {8c};
\end{tikzpicture}}&\resizebox{!}{1cm}{\begin{tikzpicture}
\filldraw [gray] 
(-.5,-.5) circle (2/2pt) (0,-.5) circle (2/2pt) (.5,-.5) circle (2/2pt)
(-.5,0) circle (2/2pt)  (.5,0) circle (2/2pt)
(-.5,.5) circle (2/2pt) (0,.5) circle (2/2pt) (.5,.5) circle (2/2pt)
(1,-.5) circle (2/2pt) (1,0) circle (2/2pt) (1,.5) circle (2/2pt)
(-.5,1) circle (2/2pt) (0,1) circle (2/2pt) (.5,1) circle (2/2pt)
(1,1) circle (2/2pt);
\draw (-.5,1) -- (-.5,-.5) -- (1,-.5) -- (-.5,1);  
\draw (0,0) node {9a};
\end{tikzpicture}}
\end{tabular}
 \end{center}
Fix a reflexive polygon $\Pg$, let $(0,0)$ be the internal lattice point and $v_1,\dots,v_k$ be the lattice points on the boundary of the polygon in cyclic order.
From this polygon we can construct a toric fan
\[
\{0, \cone{v_1}\,\dots, \cone{v_k}, \cone{v_1,v_2},\dots, \cone{v_k,v_1} \} 
\]
where $[u_1,\dots,u_l]$ is shorthand for $\R^+u_1+\dots+\R^+u_l$. 
This fan define a projective smooth toric surface $X_\Pg$. 
This surface is a Fano variety\footnote{A smooth variety is Fano if its anticanonical bundle is ample.} if all $v_i$ are corners of the polygon 
and a weak Fano variety otherwise. 

We can associate a sequence of numbers
\[
 (a_1,\dots, a_k) \text{ such that } v_{i-1}+a_iv_i+v_{i+1}=0.
\]
to this fan and up to cyclic shifts and inversion of the order this sequence determines the polygon up to affine transformations 
and the toric variety up to isomorphism.

In \cite{perlinghille} Hille and Perling studied
full cyclic strongly exceptional sequences of line bundles on weak Fano surfaces.
These are infinite sequences of line bundles $\dots, \cLL_i, \cLL_{i+1},\dots$ such that 
\begin{itemize}
 \item $\Ext^r(\cLL_i,\cLL_j)=\Ext^r(\cLL_j,\cLL_i)=0$ if $r>0$ and $i\le j< i+k$,
 \item $\Hom(\cLL_i,\cLL_j)=0$ if $i>j$,
 \item $\cLL_{i+k} = \cLL_{i}\otimes \cKK^{-1}$.
\end{itemize}
Where $\cKK$ is the canonical bundle and $k$ is the rank of the Grothendieck group, which is the same as the number of vectors
$v_i$.

Hille and Perling classified these full cyclic strongly exceptional sequences and 
proved a remarkable and strange result: 
\begin{theorem}[Hille, Perling \cite{perlinghille}]
Given a cyclic full strongly exceptional sequence $(\cLL_i)$ on a toric surface,
the sequence of numbers
\[
(b_1,\dots,b_k) := (\dim\Hom(\cLL_{i},\cLL_{i+1})-2,\dots, \dim\Hom(\cLL_{i+k-1},\cLL_{i+k})-2)
\]
corresponds to the sequence of a new weak Fano surface.
\end{theorem}
The origin and interpretation of this new surface seem at first mysterious, but recent developements
in the study of mirror symmetry for punctured Riemann surfaces \cite{abouzaid} and its relations to dimer models
\cite{Bocklandtmirror} shed new light on this. 

To every cyclic full strongly exceptional sequence on a toric weak Fano surface, 
one can associate a consistent dimer model, which is a quiver embedded in a Riemann surface. 
The dimer contains enough information to recover both the surface and
the exceptional collection. More precisely 
the $(a_i)$-sequence and the $(b_i)$-sequence can be determined from the dimer model but not the other way round.
 
A dimer model can be used to define two categories: a Fukaya category and a category of matrix factorizations. 
In \cite{Bocklandtmirror} it is shown that there is a duality on the set of dimer models, 
such that (under certain consistency conditions) the category of matrix factorizations is ($\cA_\infty$)-equivalent
to the Fukaya category of the dual dimer. This duality gives a combinatorial description of mirror symmetry for punctured
Riemann surfaces. 

The main result of this paper is that this duality acts as an involution on set of dimers coming from full strongly exceptional
sequences of line bundles  on weak toric Fano surfaces and interchanges
the $(a_i)$-sequence and the $(b_i)$-sequence.
This is an extension of Hille's and Perling's result in the following way: the dimer duality does not
only associate to a cyclic full strongly exceptional
sequence of line bundles a new toric weak Fano surface but it also equipes it with a 
new full strongly exceptional sequence. Moreover a this process is a duality and therefore
the toric system of the new exceptional sequence will give us the original toric Fano back.

The write-up of the paper is as follows. We start with a little introduction to dimer models, then we explain the phenomenon of dimer
duality and its relation to mirror symmetry. In section \ref{mirrorFano} we apply this to the special situation of weak toric Fano varieties and prove our main theorem.
We end with an illustration of the duality for reflexive polygons with 8 lattice points on the boundary.

\section{Dimer models}

A \emph{quiver} $Q$ is an oriented graph. We denote the set
of vertices by $Q_0$, the set of arrows by $Q_1$ and the maps $h,t$
assign to each arrow its head and tail. Paths are defined in the usual way as sequences of arrows $a_k\cdots a_0$
such that $t(a_i)=h(a_{i+1})$. A path is \emph{cyclic} if its head and tail coincide and a \emph{cycle} is the equivalence class of a cyclic path up to cyclic permutation.
A \emph{trivial path} is just a vertex. The \emph{path algebra} $\C Q$ is the vector
space spanned by the paths with as product the concatenation of paths.

Let $S$ be a compact orientable surface without boundary. A quiver is called \emph{embedded in $S$} if 
the vertices are a subset of $S$ and arrowa can be seen as smooth curves connecting their heads and tails, such that the only 
intersections occur end points. The surface in which a quiver $\qpol$ is embedded will often be denoted by $|\qpol|$.

An embedded quiver is called a \emph{dimer model}
the complement of the arrows is a disjoint union of opend discs, each bounded by a cyclic path of length at least $3$.
We will call the anticlockwise boundary cycles the positive cycles and group them in a set $\qpol_2^+$ and the clockwise cycles will be grouped
in $\qpol_2^-$. Note that the dimer and its surface are completely determined by the sets $\qpol_2^\pm$. 
For more information on dimers we refer to \cite{Broomhead},\cite{Bocklandtcons} and \cite{Kenyon}.

\begin{example}\label{dimex}
We give 3 examples of dimer models. The first 2 are embedded in a torus, 
the last in a double torus.
Arrows and vertices with the same label are identified.
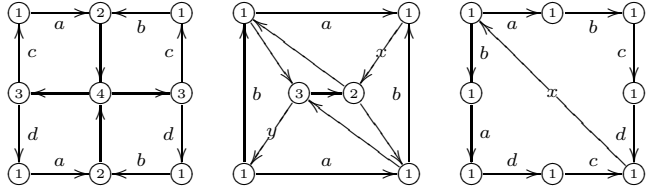
\[
\hspace{.5cm}
\xymatrix@C=.75cm@R=.75cm{
\vtx{1}\ar[r]_{a}&\vtx{2}\ar[d]&\vtx{1}\ar[l]^b\\
\vtx{3}\ar[u]_c\ar[d]^d&\vtx{4}\ar[l]\ar[r]&\vtx{3}\ar[u]^c\ar[d]_d\\
\vtx{1}\ar[r]^{a}&\vtx{2}\ar[u]&\vtx{1}\ar[l]_b
}
\hspace{.5cm}
\xymatrix@C=.4cm@R=.75cm{
\vtx{1}\ar[rrr]_a\ar[dr]&&&\vtx{1}\ar[ld]|x\\
&\vtx{3}\ar[r]\ar[ld]|y&\vtx{2}\ar[ull]\ar[dr]&\\
\vtx{1}\ar[rrr]^a\ar[uu]_b&&&\vtx{1}\ar[ull]\ar[uu]^b
}
\hspace{.5cm}
\xymatrix@C=.75cm@R=.75cm{
\vtx{1}\ar[r]_{a}\ar[d]^{b}&\vtx{1}\ar[r]_b&\vtx{1}\ar[d]_c\\
\vtx{1}\ar[d]^a&&\vtx{1}\ar[d]_d\\
\vtx{1}\ar[r]^{d}&\vtx{1}\ar[r]^c&\vtx{1}\ar[uull]|x
}
\]
\end{example}

The \emph{Jacobi algebra of a dimer model} is the quotient of the path algebra by the ideal generated by relations
of the form $r_a := r_+-r_-$ where $r_+a \in \qpol_2^+$ and $r_-a\in \qpol_2^-$ for some arrow $a\in \qpol_1$:
\[
 \cJA(\qpol) := \frac{\C \qpol}{\<r_a| a \in \qpol_1\>}.
\]
Every Jacobi algebra has a central element $$\ell=\sum_{v\in \qpol_0} c_v$$ 
where for each vertex $c_v$ is a cyclic path with $h(c_v)=t(c_v)=v$ that forms a cycle in $Q_2$.
Using the relations one can show that this is indeed central and does not depend on the cycles we chose the sum.

Fix a dimer model $\qpol$ and denote its universal cover\footnote{The universal cover of $\qpol$ embedded in $S$ is obtained by lifting all arrows and vertices in all possible ways to the universal cover of $S$. This may result in an infinite quiver.} 
by $\tilde \qpol$. 
For any arrow $\tilde a\in \tilde \qpol_1$ 
we can construct its \emph{zig ray} $\pZ_{\tilde a}^+$. This is an infinite path
\[
\dots \tilde a_2\tilde a_1\tilde a_0
\]
such that $\tilde a_0=\tilde a$ and $\tilde a_{i+1}\tilde a_{i}$ sits in a positive cycle if $i$ is even and in a negative cycle if $i$ is odd.
Similarly the \emph{zag ray} $\pZ_{\tilde a}^-$ is the path where  $\tilde a_{i+1}\tilde a_{i}$ sits in a positive cycle if $i$ is odd and in a negative cycle if $i$ is even.
The projection of a zig or a zag ray down to $Q$ will give us a cyclic path because $Q$ is finite. Such a cyclic
path will be called a \emph{zigzag cycle}.
A dimer model is called \emph{zigzag consistent} if for every arrow $\tilde a$ the zig and the zag ray only meet in $\tilde a$:
\[
 (\pZ_{\tilde a}^-)_i= (\pZ_{\tilde a}^+)_j \implies i=j=0.
\]
In example \ref{dimex} the first and third quiver are consistent, while the second 
quiver is not because $(\pZ_{\tilde x}^-)_3= (\pZ_{\tilde x}^+)_3=\tilde y$.

For dimer models on a torus there are 2 extra characterizations of consistency, which we will use later on.
\begin{theorem}[Bocklandt\cite{Bocklandtcons}]
For a dimer model $\qpol$ on a torus the following are equivalent:
\begin{itemize}
 \item $\qpol$ is zigzag consistent.
 \item $\qpol$ admits a consistent $\cR$-charge.
This is a map $\cR:\qpol_1\to (0,2)$ such that every positive or negative cycle has degree $2$ and for every vertex $v$
we have
\[
 \sum_{h(a)=v} (1-\cR_a) + \sum_{t(a)=v} (1-\cR_a) = 2
\]
 \item $\cJA(\qpol)$ embeds in $\widehat{\cJA}(\qpol) := \cJA(\qpol)\otimes_{\C[\ell]}\C[\ell,\ell^{-1}]$.
\end{itemize}
\end{theorem}
\begin{remark}
We can see $\widehat{\cJA}(\qpol)$ as the algebra in which all arrows get an inverse because 
$a^{-1}=p\ell^{-1}$ if $ap$ is a positive cycle.
We will call paths and cycles in this algebra \emph{weak paths} and \emph{weak cycles}.
If we want to stress that a path is not weak, we will call it \emph{real}.
\end{remark}

Another ingredient we need are \emph{perfect matchings}. These are subsets $\PM\subset \qpol_1$ such
that every positive and negative cycle contains precisely one arrow of $\PM$.
Every perfect matching can also be seen as a degree function on $\widehat{\cJA}(\qpol)$ 
that gives $a$ degree $1$ if $a\in \PM$ and degree $0$ otherwise. We will write $\PM(p)$ for the degree of a weak path $p$.
Note that for the element $\ell$ we have $\PM(\ell)=1$. 

Fix a vertex $\triv$, which we will call the trivial vertex, and two weak cycles $x,y$ that span the homology 
of the torus.
To every perfect matching $\PM$ we can associate a point $(\PM(x),\PM(y),\PM(\ell))\in \Z^3$. 
A set of matchings $\{\PM_1,\dots,\PM_u\}$ is called $\triv$-stable\footnote{The notion of $\triv$-stable coincides with $\theta$-stable if $\theta$ is negative on $\triv$ and positive on all other vertices.} 
if there is a real path from
$\triv$ to every other vertex that has $\PM_i$-degree zero for all matchings in the set.
For each stable set $S$ we can define a cone $\sigma_S=\sum_{\PM \in S} \R^+(\PM(x),\PM(y),\PM(\ell))$

The technique of perfect matchings can be used to relate dimer models to the geometry of crepant resolutions of Gorenstein singularities.
\begin{theorem}[Ishii-Ueda\cite{IU}-Mozgovoy-Bender\cite{Moz,MB}]\label{pmsmain}
If $\qpol$ is a consistent dimer on a torus then
\begin{enumerate}
 \item The collection of cones $\sigma_S$ where $S$ is a stable set of matchings forms a fan and the toric variety 
of this fan, $\tilde X$ is a crepant resolution of $X=\Spec Z(A)$.
 \item If we intersect the fan with the plane at height $z=1$, we get a convex polygon $\Pg$, which is subdivided
in elementary triangles and on each integral point of the polygon sits a unique stable perfect matching.
These lattice points form a basis for the toric divisors of $\tilde X$. So any line $\Z$-linear combination
of stable perfect matchings gives us a line bundle.
 \item Fix a set of paths $\{p_v\}$ from $\triv$ to every other vertex $v$.
The direct sum $\cTT$ of the line bundles $\cLL_v$ with divisors $\sum \PM_i(p_v)\PM_i$ (the sum runs over all stable perfect matchings)
is a tilting bundle on $\tilde X$ and $\End(\cTT)=\cJA(\qpol)$.  
\end{enumerate}
\end{theorem}
This theorem has some implications which we will need further on.

\begin{lemma}\label{pmscor}
If $\qpol$ is a consistent dimer on a torus then
\begin{enumerate}
 \item The number of vertices in $\qpol$ is the number of elementary triangles in $\Pg$. 
 \item The number of zigzag paths in $\qpol$ is the number of elementary line segments on the boundary of $\Pg$.
More precisely, if \{$\PM_1,\PM_2\}$ is a stable set of two perfect matchings on the boundary then 
the arrows contained in either $\PM_1$ or $\PM_2$ but not in both form a zigzag path.
All zigzag paths arise in this way.
 \item For a given arrow $a$, the stable perfect matchings on the boundary that contain $a$ 
are precisely the ones that lie inbetween
the elementary line segments of its zig and its zag path. 
 \item For each weak path $p$ in $\qpol$ we can split the stable matchings in two sets $M^+=\{\PM|\PM(p)\ge 0\}$ 
and $M^-:=\{\PM|\PM(p)< 0\}$. The subset of triangles, line segments and lattice points spanned by lattice points
in $\Pg$ that correspond to stable sets in $M^\pm$ form a simply connected simplicial set.
  \item For each weak path $p$ in $\qpol$ we can split the stable \emph{boundary matchings} in two sets $B^\pm:=M^\pm \cap \partial \Pg$. The subset of line segments and lattice points spanned by lattice points
in $\partial \Pg$ that correspond to stable sets in $B^\pm$ form a connected simplicial set (i.e a circle segment or the whole circle).
\end{enumerate}
\end{lemma}
\begin{proof}
The first statement follows because the derived equivalence gives an equality between the ranks of the Grotendieck groups, 
which for the quiver is the number of vertices and for the crepant resolution the number of elementary triangles 
needed to subdivide $\Pg$.

A proof of the second and third statement can be found in \cite{BCQ}.

The fourth statement can be proved using simplicial homology. Let $p$ be any weak path in the dimer with $h(p)=v$ and $t(p)=w$ then
$\sum \PM_i(p)\PM_i$ is a line bundle equivalent to $\cLL_v\cLL_w^{-1}$. Because $\cTT$ is tilting we have that
$\cLL_v\cLL_w^{-1}$ has no higher homology. The homology can be calculated from the complex $\cF^\bullet,\delta$
with
$$\cF^r := \bigoplus_{|S|=r}\C[X^{m_1}Y^{m_2}Z^{m_3}| m_1\PM(x)+m_2\PM(y)+m_3\ge -\PM(p) \forall \PM \in S]$$
where the sums are over the stable sets of matchings and $$\delta (X^{m_1}Y^{m_2}Z^{m_3})_S=\sum_{\PM \in S} \pm (X^{m_1}Y^{m_2}Z^{m_3})_{S\setminus\{\PM\}}$$
is the boundary map between these simplicial sets. 

If we look at the summand corresponding to $X^{m_1}Y^{m_2}Z^{m_3}=1$ we get that this complex calculates the 
simplicial homology of 
the simplicial subcomplex containing only the $S\subset M^+$. This is acyclic if and only if
the subcomplex is simply connected. To get the statement for $M^-$ we need to look at the weak path $p^{-1}\ell^{-1}$.

The fifth statement is an easy consequence of the fourth: if $B^+$ is not connected then either $M^+$ or $M^-$ is not connected. 
\end{proof}

\section{Mirror symmetry for dimers}

In general homological mirror symmetry conjectures an equivalence
between two categories: one constructed from algebraic geometry and one from symplectic geometry.
In the case we will be considering both categories can be constructed explicitely from a dimer model.
In this section we will summarize the main results of \cite{Bocklandtmirror}.

\subsection{The Fukaya category}

If $\qpol$ is a dimer model we can look at its wrapped Fukaya category\footnote{For more information on general wrapped Fakaya categories we refer to \cite{abouzaidwrapped}.}.  
Objects in this category are the arrows of the quiver which we consider as Langrangian submanifolds
of the underlying surface punctured by the vertices.
Morphisms between Lagrangians are given by time-one flow curves of a hamiltonian flow on the surface that
connect these Lagrangians. The products are given by counting certain maps from the disc to the surface
such that the boundary of this disc lies on the Lagrangians and the time-one flow curves.
This produces an $\cA_\infty$-category $\fuk(\qpol)$.

\subsection{Matrix factorizations}

To each arrow $a$ in a dimer we can associate a matrix factorization of $\ell\in \cJA(\qpol)$, 
which is a diagram of the form
\[
\bar P_a :=  \xymatrix{\cJA(\qpol)h(a) \ar@<.5ex>[r]^{a}&\cJA(\qpol)t(a) \ar@<.5ex>[l]^{\bar a}}
\]
where $\bar a$ is defined such that $a\bar a\in \qpol_2^+$. Note that $\bar P_a$ can also be seen as a
$\Z_2$ graded projective $\cJA(\qpol)$-module with a curved differential $d$ such that $d^2=\ell$.

Given $2$ such matrix factorizations $\bar P$, $\bar Q$, the space $\Hom_{\cJA(\qpol)}(\bar P,\bar Q)$ becomes
equiped with an ordinary (noncurved) differential $\delta$. The category $\H\mf(\qpol)$ contains as objects the matrix factorizations $\bar P_a$ and
as hom-spaces the homology of $\delta$ on $\Hom(\bar P,\bar Q)$. The dg-structure on the $\Hom_{\cJA(\qpol)}(\bar P,\bar Q)$
can be turned into an $\cA_\infty$-structure on $\H\mf(\qpol)$. 

\subsection{Dimer duality}

The two categories we defined above can be related by a certain duality on the level of dimers.
Let $\qpol$ be any, not necessarily consistent, dimer. 
We define its mirror dimer $\mirror{\qpol}$ as follows

\begin{enumerate}
 \item The vertices of $\mirror{\qpol}$ are the zigzag cycles of $\qpol$.
 \item The arrows of $\mirror{\qpol}$ are the arrows of $\qpol$, $h(a)$ is the zigzag cycle coming from the zig ray, and $t(a)$
is the cycle coming from the zag ray.
 \item The positive faces of $\mirror{\qpol}$ are the positive faces of $\qpol$.
 \item The negative faces of $\mirror{\qpol}$ are the negative faces of $\qpol$ in reverse order. 
\end{enumerate}
Some relevant examples of dimer duality can be found in the table of the last section of the paper.
\begin{remark}
The dual can also be obtained by cutting the dimer along the arrows, flipping over the clockwise faces, reversing their arrows and gluing everything back again.
This construction is basically the same construction that was introduced by Feng, He, Kennaway and Vafa in 
\cite{quivering} applied to all possible dimers. It is also important to note that the genus of the surface in which the dual dimer lives often differs 
from the genus of the surface of the original dimer.
\end{remark}

\begin{theorem}\cite{Bocklandtmirror}
If $\qpol$ is a consistent dimer then the categories $\H\mf(\qpol)$ and $\fuk(\mirror{\qpol})$ are
$\cA_\infty$-isomorphic.
\end{theorem}

An $\cA_\infty$-category can be completed by adding twisted objects, which are morally complexes of the old ones (see \cite{Keller}).
We also need to add in projectors to ensure the category is idempotent complete.
If we apply this process to both categories above we get two categories $\cD^\pi\Fuk(\qpol)$ and $\cD^\pi\mf(\qpol)$.
The first one only depends on the surface in which $\qpol$ is embedded and the number of punctures (=vertices in $\qpol$).
We call it the the idempotent completed wrapped Fukaya category of the punctured surface.

If $\qpol$ sits on a torus then, by a theorem of Ishii and Ueda \cite{IU}, the second category is equivalent to the idempotent completion of
a category of singularities 
\[
 \cD\mathtt{Sing} f^{-1}(0) := \frac{\cD^b\Coh f^{-1}(0)}{\mathtt{Perf} f^{-1}(0)}
\]
where $f:\tilde X\to \C:p\mapsto \ell(p)$ and $\tilde X$ is a crepant resolution of $\Spec Z(\cJA(\qpol)$.

In this way we recover a the following version of mirror symmetry
\begin{theorem}\cite{Bocklandtmirror}
If $\qpol$ is a consistent dimer on a torus and then the idempotent completion of $\cD^b\mathtt{Sing} f^{-1}(0)$ is 
$\cA_\infty$-isomorphic to the idempotent completed wrapped Fukaya category of a surface with genus $\frac 12(2-\#\mirror{\qpol}_0 +\#\qpol_0)$
and $\#\mirror{\qpol}_0$ punctures.
\end{theorem}

\section{Weak Fano dimers}\label{mirrorFano}

\subsection{From surface to dimer}
Now we return to the setting of weak toric Fano varieties. Let $v_1,\dots,v_k \in \Z^2$ be the vectors in cyclic order 
that define
a weak toric Fano surface $X$ and let $(\cLL_i)_{i\in \Z}$ be a cyclic full strongly exceptional sequence on $X$.
We can represent each $\cLL_i$ by a divisor $l_{i1}E_1+\dots+l_{ik}E_k$. Here $E_j$ is the divisor corresponding to the vector $v_j$ and
$\ccE_i$ will be its corresponding line bundle.

How do we get a dimer model out of these data?
Let $\tilde Y$ be the total space of the canonical bundle $\cKK$ on $X$ and denote the natural projection by $\pi:\tilde Y\to X$.
From toric geometry we know that the fan of $Y$ can be constructed from the fan of $X$ in the following way.
We lift every vector $v_i$ to a vector $\tilde v_i=(v_{i1},v_{i2},1)$. These points form a polygon $\Pg$
in the plane with third coordinate equal to $1$. Let $z=(0,0,1)$ be the unique internal lattice point of this Polygon. 
The maximal dimensional cones of the fan of $\tilde Y$ are then  
$$
[\tilde v_1,\tilde v_{2},z], [\tilde v_2,\tilde v_{3},z],\dots [\tilde v_{k-1},\tilde v_{k},z], [\tilde v_k,\tilde v_{1},z].
$$
Because all vectors $\tilde v_i$ and $z$ lie in the same plane, $\tilde Y$ is a local Calabi-Yau 3-fold.
Moreover, this variety is a crepant resolution of affine variety $Y$ generated by the cone $[\tilde v_1,\dots,\tilde v_k]$.
To generate this cone we only need the $\tilde v_i$ that lie on the corners of the polygon. In terms of the $(a_i)$-sequence 
of the polygon these are $\tilde v_i$ for which $a_i\ne -2$. This gives us the following formula for the coordinate ring of $Y$.
\[
 \C[Y] := \C[x^{m_1}y^{m_2}z^{m_3}|  \forall \mu: a_\mu=-2:\<m,\tilde v_\mu\>\ge 0 ].
\]

A theorem of Bridgeland \cite{bridgeland} states that if $(\cLL_i)_{i\in \Z}$ is 
a cyclic full strongly exceptional sequence on $X$ then for any $i\in \Z$ the direct sum
\[
\cTT = \bigoplus_{j=i}^{i+k-1}\pi^*\cLL_j
\]
forms a tilting bundle on $\tilde Y$. Different choices of $i$ will give isomorphic tilting bundles.
The Picard group of $\tilde Y$ is generated by the toric divisors $\tilde E_i$ corresponding to the $\tilde v_i$ and an extra
divisor $Z$ coming from $z$. The pullback of a line bundle over $X$  with divisor 
$l_{i1}E_1+\dots+l_{ik}E_k$ results in a line bundle with divisor $l_{i1}\tilde E_1+\dots+l_{ik}\tilde E_k$.

Because $\tilde Y$ is a toric Calabi-Yau-3 variety and $\cTT$ is a direct sum of line bundles, $\cB$ is a toric Calabi-Yau-3 order
in the sense of the sense of \cite{Bocklandtqp}. This implies $\cB$ is the Jacobi algebra of a consistent dimer model 
$\qpol$ on a torus and a noncommutative crepant resolution of $Y$ in the sense of \cite{VandenBergh}. 
This is an endomorphism ring of reflexive $\C[Y]$-modules with global dimension equal to the dimension of $Y$.
For every $\pi^*\cLL_j$ we get a reflexive $\C[Y]$-module
\[
 L_j := \C[x^{m_1}y^{m_2}z^{m_3}|  \forall \mu: a_\mu=-2:\<m,\tilde v_\mu\>+l_{j\mu}\ge 0 ].
\]
The endomorphism ring of the direct sum of these reflexives 
$\End_{\C[Y]}\bigoplus_{j=i}^{i+k-1}L_j$ is isomorphic to $\cB$.

Analoguously to \cite{Bocklandtnccrs}, we can construct the dimer for $\cB$ in the following way. 
Let $u$ be the number of corner vertices of the polygon and
identify the lattice $\Z^u\subset \R^u$ with the set of reflexive $\C[Y]$-modules of the form
\[
 T_{a} := \C[x^{m_1}y^{m_2}z^{m_3}|  \forall \mu: a_\mu=-2:\<m,\tilde v_\mu\>+a_{j\mu}\ge 0 ] \text{ where }a\in \Z^u.
\]
Now let $\tilde Q_0$ be the subset of $\Z^u$ that corresponds to modules isomorphic to $L_i$ for
some $i$. Draw an arrow between $a$ and $b$ in $\tilde Q_0$ if $T_a\subset T_b$
(or equivalently $\forall i\le u: a_i\le b_i$) and there is no other $c\in \tilde Q_0$ with $T_a\subset T_c\subset T_b$.
Factor out the equivalence relation generated by $a\equiv b \iff  T_a\cong T_b$.
This projects our infinite quiver down to a finite quiver $\qpol$. The positive and negative cycles of 
$\qpol$ are all paths of $\Z^u$-degree $(1,\dots,1)$.
 
\begin{example}[The projective plane blown up in 3 points]
The toric fan of this Fano surface has six 2-dimensional cones as shown below. 
The full cyclic strongly exceptional sequence is obtained by extending the 6 line bundles below
to an infinite sequence by tensoring with powers of $\cKK$. An arrow corresponding to an embedding $T_a\subset T_b$ is
labeled by the vector $b-a$. 
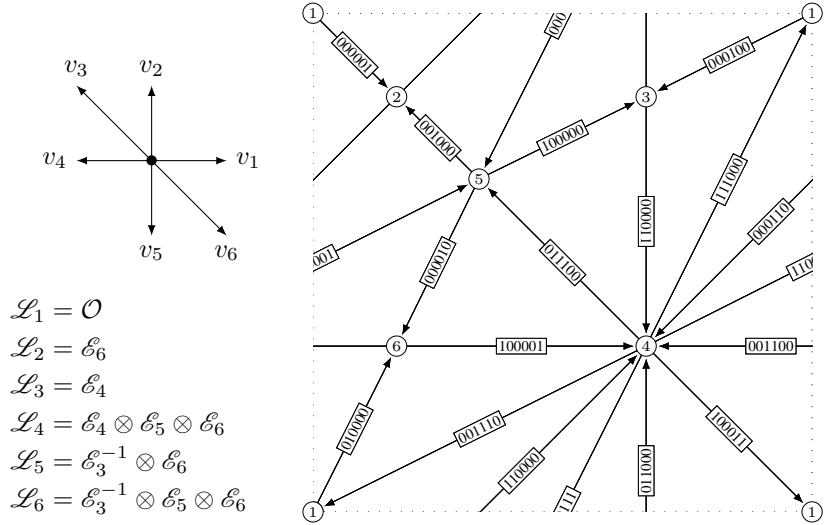
\begin{center}
\begin{tabular}{cc}
\begin{tikzpicture} 
\draw[-latex] (0,0) -- (0,1) node[above] {$v_2$};
\draw[-latex] (0,0) -- (-1,1) node[above] {$v_3$};
\draw[-latex] (0,0) -- (-1,0) node[left] {$v_4$};
\draw[-latex] (0,0) -- (0,-1) node[below] {$v_5$};
\draw[-latex] (0,0) -- (1,-1) node[below] {$v_6$};
\draw[-latex] (0,0) -- (1,0) node[right] {$v_1$};
\draw (0,0) node{$\bullet$};
\draw (-2,-2) node[right]{$\cLL_1=\cO$};
\draw (-2,-2.5) node[right]{$\cLL_2=\ccE_6$};
\draw (-2,-3) node[right]{$\cLL_3=\ccE_4$};
\draw (-2,-3.5) node[right]{$\cLL_4=\ccE_4\otimes \ccE_5\otimes \ccE_6$};
\draw (-2,-4) node[right]{$\cLL_5=\ccE_3^{-1}\otimes \ccE_6$};
\draw (-2,-4.5) node[right]{$\cLL_6=\ccE_3^{-1}\otimes \ccE_5\otimes \ccE_6$};
\end{tikzpicture}&
\resizebox{7cm}{!}{
\begin{tikzpicture} 
\begin{scope}\clip (20/2pt,20/2pt) rectangle (420/2pt,420/2pt);
\draw[loosely dotted] (20/2pt,20/2pt) rectangle (420/2pt,420/2pt);
\draw [-latex,shorten >=5pt] (20/2pt,20/2pt) to node [rectangle,draw,fill=white,sloped,inner sep=1pt] {{\tiny 000100}}  (-113/2pt,-47/2pt); 
\draw [-latex,shorten >=5pt] (20/2pt,20/2pt) to node [rectangle,draw,fill=white,sloped,inner sep=1pt] {{\tiny 000001}}  (87/2pt,-47/2pt); 
\draw [-latex,shorten >=5pt] (287/2pt,153/2pt) to node [rectangle,draw,fill=white,sloped,inner sep=1pt] {{\tiny 000111}}  (153/2pt,-113/2pt); 
\draw [-latex,shorten >=5pt] (87/2pt,353/2pt) to node [rectangle,draw,fill=white,sloped,inner sep=1pt] {{\tiny 000110}}  (-113/2pt,153/2pt); 
\draw [-latex,shorten >=5pt] (87/2pt,153/2pt) to node [rectangle,draw,fill=white,sloped,inner sep=1pt] {{\tiny 001100}}  (-113/2pt,153/2pt); 
\draw [-latex,shorten >=5pt] (20/2pt,420/2pt) to node [rectangle,draw,fill=white,sloped,inner sep=1pt] {{\tiny 000001}}  (87/2pt,353/2pt); 
\draw [-latex,shorten >=5pt] (420/2pt,420/2pt) to node [rectangle,draw,fill=white,sloped,inner sep=1pt] {{\tiny 000100}}  (287/2pt,353/2pt); 
\draw [-latex,shorten >=5pt] (20/2pt,20/2pt) to node [rectangle,draw,fill=white,sloped,inner sep=1pt] {{\tiny 010000}}  (87/2pt,153/2pt); 
\draw [-latex,shorten >=5pt] (487/2pt,353/2pt) to node [rectangle,draw,fill=white,sloped,inner sep=1pt] {{\tiny 000110}}  (287/2pt,153/2pt); 
\draw [-latex,shorten >=5pt] (87/2pt,-47/2pt) to node [rectangle,draw,fill=white,sloped,inner sep=1pt] {{\tiny 110000}}  (287/2pt,153/2pt); 
\draw [-latex,shorten >=5pt] (287/2pt,353/2pt) to node [rectangle,draw,fill=white,sloped,inner sep=1pt] {{\tiny 000011}}  (287/2pt,153/2pt); 
\draw [-latex,shorten >=5pt] (287/2pt,-47/2pt) to node [rectangle,draw,fill=white,sloped,inner sep=1pt] {{\tiny 011000}}  (287/2pt,153/2pt); 
\draw [-latex,shorten >=5pt] (287/2pt,153/2pt) to node [rectangle,draw,fill=white,sloped,inner sep=1pt] {{\tiny 001110}}  (20/2pt,20/2pt); 
\draw [-latex,shorten >=5pt] (-113/2pt,153/2pt) to node [rectangle,draw,fill=white,sloped,inner sep=1pt] {{\tiny 100011}}  (20/2pt,20/2pt); 
\draw [-latex,shorten >=5pt] (-113/2pt,-247/2pt) to node [rectangle,draw,fill=white,sloped,inner sep=1pt] {{\tiny 111000}}  (20/2pt,20/2pt); 
\draw [-latex,shorten >=5pt] (287/2pt,553/2pt) to node [rectangle,draw,fill=white,sloped,inner sep=1pt] {{\tiny 000111}}  (153/2pt,287/2pt); 
\draw [-latex,shorten >=5pt] (287/2pt,153/2pt) to node [rectangle,draw,fill=white,sloped,inner sep=1pt] {{\tiny 011100}}  (153/2pt,287/2pt); 
\draw [-latex,shorten >=5pt] (-113/2pt,153/2pt) to node [rectangle,draw,fill=white,sloped,inner sep=1pt] {{\tiny 110001}}  (153/2pt,287/2pt); 
\draw [-latex,shorten >=5pt] (153/2pt,287/2pt) to node [rectangle,draw,fill=white,sloped,inner sep=1pt] {{\tiny 001000}}  (87/2pt,353/2pt); 
\draw [-latex,shorten >=5pt] (153/2pt,287/2pt) to node [rectangle,draw,fill=white,sloped,inner sep=1pt] {{\tiny 100000}}  (287/2pt,353/2pt); 
\draw [-latex,shorten >=5pt] (153/2pt,287/2pt) to node [rectangle,draw,fill=white,sloped,inner sep=1pt] {{\tiny 000010}}  (87/2pt,153/2pt); 
\draw [-latex,shorten >=5pt] (487/2pt,153/2pt) to node [rectangle,draw,fill=white,sloped,inner sep=1pt] {{\tiny 001100}}  (287/2pt,153/2pt); 
\draw [-latex,shorten >=5pt] (87/2pt,153/2pt) to node [rectangle,draw,fill=white,sloped,inner sep=1pt] {{\tiny 100001}}  (287/2pt,153/2pt); 
\draw [-latex,shorten >=5pt] (287/2pt,153/2pt) to node [rectangle,draw,fill=white,sloped,inner sep=1pt] {{\tiny 100011}}  (420/2pt,20/2pt); 
\draw [-latex,shorten >=5pt] (287/2pt,153/2pt) to node [rectangle,draw,fill=white,sloped,inner sep=1pt] {{\tiny 110001}}  (553/2pt,287/2pt); 
\draw [-latex,shorten >=5pt] (87/2pt,353/2pt) to node [rectangle,draw,fill=white,sloped,inner sep=1pt] {{\tiny 110000}}  (287/2pt,553/2pt); 
\draw [-latex,shorten >=5pt] (287/2pt,353/2pt) to node [rectangle,draw,fill=white,sloped,inner sep=1pt] {{\tiny 011000}}  (287/2pt,553/2pt); 
\draw [-latex,shorten >=5pt] (287/2pt,153/2pt) to node [rectangle,draw,fill=white,sloped,inner sep=1pt] {{\tiny 111000}}  (420/2pt,420/2pt); 
\draw [-latex,shorten >=5pt] (20/2pt,20/2pt) to node [rectangle,draw,fill=white,sloped,inner sep=1pt] {{\tiny 000100}}  (-113/2pt,-47/2pt); 
\draw [-latex,shorten >=5pt] (20/2pt,20/2pt) to node [rectangle,draw,fill=white,sloped,inner sep=1pt] {{\tiny 000001}}  (87/2pt,-47/2pt); 
\draw [-latex,shorten >=5pt] (287/2pt,153/2pt) to node [rectangle,draw,fill=white,sloped,inner sep=1pt] {{\tiny 000111}}  (153/2pt,-113/2pt); 
\draw [-latex,shorten >=5pt] (87/2pt,353/2pt) to node [rectangle,draw,fill=white,sloped,inner sep=1pt] {{\tiny 000110}}  (-113/2pt,153/2pt); 
\draw [-latex,shorten >=5pt] (87/2pt,153/2pt) to node [rectangle,draw,fill=white,sloped,inner sep=1pt] {{\tiny 001100}}  (-113/2pt,153/2pt); 
\draw [-latex,shorten >=5pt] (20/2pt,420/2pt) to node [rectangle,draw,fill=white,sloped,inner sep=1pt] {{\tiny 000001}}  (87/2pt,353/2pt); 
\draw [-latex,shorten >=5pt] (420/2pt,420/2pt) to node [rectangle,draw,fill=white,sloped,inner sep=1pt] {{\tiny 000100}}  (287/2pt,353/2pt); 
\draw [-latex,shorten >=5pt] (20/2pt,20/2pt) to node [rectangle,draw,fill=white,sloped,inner sep=1pt] {{\tiny 010000}}  (87/2pt,153/2pt); 
\draw [-latex,shorten >=5pt] (487/2pt,353/2pt) to node [rectangle,draw,fill=white,sloped,inner sep=1pt] {{\tiny 000110}}  (287/2pt,153/2pt); 
\draw [-latex,shorten >=5pt] (87/2pt,-47/2pt) to node [rectangle,draw,fill=white,sloped,inner sep=1pt] {{\tiny 110000}}  (287/2pt,153/2pt); 
\draw [-latex,shorten >=5pt] (287/2pt,353/2pt) to node [rectangle,draw,fill=white,sloped,inner sep=1pt] {{\tiny 000011}}  (287/2pt,153/2pt); 
\draw [-latex,shorten >=5pt] (287/2pt,-47/2pt) to node [rectangle,draw,fill=white,sloped,inner sep=1pt] {{\tiny 011000}}  (287/2pt,153/2pt); 
\draw [-latex,shorten >=5pt] (287/2pt,153/2pt) to node [rectangle,draw,fill=white,sloped,inner sep=1pt] {{\tiny 001110}}  (20/2pt,20/2pt); 
\draw [-latex,shorten >=5pt] (-113/2pt,153/2pt) to node [rectangle,draw,fill=white,sloped,inner sep=1pt] {{\tiny 100011}}  (20/2pt,20/2pt); 
\draw [-latex,shorten >=5pt] (-113/2pt,-247/2pt) to node [rectangle,draw,fill=white,sloped,inner sep=1pt] {{\tiny 111000}}  (20/2pt,20/2pt); 
\draw [-latex,shorten >=5pt] (287/2pt,553/2pt) to node [rectangle,draw,fill=white,sloped,inner sep=1pt] {{\tiny 000111}}  (153/2pt,287/2pt); 
\draw [-latex,shorten >=5pt] (287/2pt,153/2pt) to node [rectangle,draw,fill=white,sloped,inner sep=1pt] {{\tiny 011100}}  (153/2pt,287/2pt); 
\draw [-latex,shorten >=5pt] (-113/2pt,153/2pt) to node [rectangle,draw,fill=white,sloped,inner sep=1pt] {{\tiny 110001}}  (153/2pt,287/2pt); 
\draw [-latex,shorten >=5pt] (153/2pt,287/2pt) to node [rectangle,draw,fill=white,sloped,inner sep=1pt] {{\tiny 001000}}  (87/2pt,353/2pt); 
\draw [-latex,shorten >=5pt] (153/2pt,287/2pt) to node [rectangle,draw,fill=white,sloped,inner sep=1pt] {{\tiny 100000}}  (287/2pt,353/2pt); 
\draw [-latex,shorten >=5pt] (153/2pt,287/2pt) to node [rectangle,draw,fill=white,sloped,inner sep=1pt] {{\tiny 000010}}  (87/2pt,153/2pt); 
\draw [-latex,shorten >=5pt] (487/2pt,153/2pt) to node [rectangle,draw,fill=white,sloped,inner sep=1pt] {{\tiny 001100}}  (287/2pt,153/2pt); 
\draw [-latex,shorten >=5pt] (87/2pt,153/2pt) to node [rectangle,draw,fill=white,sloped,inner sep=1pt] {{\tiny 100001}}  (287/2pt,153/2pt); 
\draw [-latex,shorten >=5pt] (287/2pt,153/2pt) to node [rectangle,draw,fill=white,sloped,inner sep=1pt] {{\tiny 100011}}  (420/2pt,20/2pt); 
\draw [-latex,shorten >=5pt] (287/2pt,153/2pt) to node [rectangle,draw,fill=white,sloped,inner sep=1pt] {{\tiny 110001}}  (553/2pt,287/2pt); 
\draw [-latex,shorten >=5pt] (87/2pt,353/2pt) to node [rectangle,draw,fill=white,sloped,inner sep=1pt] {{\tiny 110000}}  (287/2pt,553/2pt); 
\draw [-latex,shorten >=5pt] (287/2pt,353/2pt) to node [rectangle,draw,fill=white,sloped,inner sep=1pt] {{\tiny 011000}}  (287/2pt,553/2pt); 
\draw [-latex,shorten >=5pt] (287/2pt,153/2pt) to node [rectangle,draw,fill=white,sloped,inner sep=1pt] {{\tiny 111000}}  (420/2pt,420/2pt); 
\draw [-latex,shorten >=5pt] (20/2pt,20/2pt) to node [rectangle,draw,fill=white,sloped,inner sep=1pt] {{\tiny 000100}}  (-113/2pt,-47/2pt); 
\draw [-latex,shorten >=5pt] (20/2pt,20/2pt) to node [rectangle,draw,fill=white,sloped,inner sep=1pt] {{\tiny 000001}}  (87/2pt,-47/2pt); 
\draw [-latex,shorten >=5pt] (287/2pt,153/2pt) to node [rectangle,draw,fill=white,sloped,inner sep=1pt] {{\tiny 000111}}  (153/2pt,-113/2pt); 
\draw [-latex,shorten >=5pt] (87/2pt,353/2pt) to node [rectangle,draw,fill=white,sloped,inner sep=1pt] {{\tiny 000110}}  (-113/2pt,153/2pt); 
\draw [-latex,shorten >=5pt] (87/2pt,153/2pt) to node [rectangle,draw,fill=white,sloped,inner sep=1pt] {{\tiny 001100}}  (-113/2pt,153/2pt); 
\draw [-latex,shorten >=5pt] (20/2pt,420/2pt) to node [rectangle,draw,fill=white,sloped,inner sep=1pt] {{\tiny 000001}}  (87/2pt,353/2pt); 
\draw [-latex,shorten >=5pt] (420/2pt,420/2pt) to node [rectangle,draw,fill=white,sloped,inner sep=1pt] {{\tiny 000100}}  (287/2pt,353/2pt); 
\draw [-latex,shorten >=5pt] (20/2pt,20/2pt) to node [rectangle,draw,fill=white,sloped,inner sep=1pt] {{\tiny 010000}}  (87/2pt,153/2pt); 
\draw [-latex,shorten >=5pt] (487/2pt,353/2pt) to node [rectangle,draw,fill=white,sloped,inner sep=1pt] {{\tiny 000110}}  (287/2pt,153/2pt); 
\draw [-latex,shorten >=5pt] (87/2pt,-47/2pt) to node [rectangle,draw,fill=white,sloped,inner sep=1pt] {{\tiny 110000}}  (287/2pt,153/2pt); 
\draw [-latex,shorten >=5pt] (287/2pt,353/2pt) to node [rectangle,draw,fill=white,sloped,inner sep=1pt] {{\tiny 000011}}  (287/2pt,153/2pt); 
\draw [-latex,shorten >=5pt] (287/2pt,-47/2pt) to node [rectangle,draw,fill=white,sloped,inner sep=1pt] {{\tiny 011000}}  (287/2pt,153/2pt); 
\draw [-latex,shorten >=5pt] (287/2pt,153/2pt) to node [rectangle,draw,fill=white,sloped,inner sep=1pt] {{\tiny 001110}}  (20/2pt,20/2pt); 
\draw [-latex,shorten >=5pt] (-113/2pt,153/2pt) to node [rectangle,draw,fill=white,sloped,inner sep=1pt] {{\tiny 100011}}  (20/2pt,20/2pt); 
\draw [-latex,shorten >=5pt] (-113/2pt,-247/2pt) to node [rectangle,draw,fill=white,sloped,inner sep=1pt] {{\tiny 111000}}  (20/2pt,20/2pt); 
\draw [-latex,shorten >=5pt] (287/2pt,553/2pt) to node [rectangle,draw,fill=white,sloped,inner sep=1pt] {{\tiny 000111}}  (153/2pt,287/2pt); 
\draw [-latex,shorten >=5pt] (287/2pt,153/2pt) to node [rectangle,draw,fill=white,sloped,inner sep=1pt] {{\tiny 011100}}  (153/2pt,287/2pt); 
\draw [-latex,shorten >=5pt] (-113/2pt,153/2pt) to node [rectangle,draw,fill=white,sloped,inner sep=1pt] {{\tiny 110001}}  (153/2pt,287/2pt); 
\draw [-latex,shorten >=5pt] (153/2pt,287/2pt) to node [rectangle,draw,fill=white,sloped,inner sep=1pt] {{\tiny 001000}}  (87/2pt,353/2pt); 
\draw [-latex,shorten >=5pt] (153/2pt,287/2pt) to node [rectangle,draw,fill=white,sloped,inner sep=1pt] {{\tiny 100000}}  (287/2pt,353/2pt); 
\draw [-latex,shorten >=5pt] (153/2pt,287/2pt) to node [rectangle,draw,fill=white,sloped,inner sep=1pt] {{\tiny 000010}}  (87/2pt,153/2pt); 
\draw [-latex,shorten >=5pt] (487/2pt,153/2pt) to node [rectangle,draw,fill=white,sloped,inner sep=1pt] {{\tiny 001100}}  (287/2pt,153/2pt); 
\draw [-latex,shorten >=5pt] (87/2pt,153/2pt) to node [rectangle,draw,fill=white,sloped,inner sep=1pt] {{\tiny 100001}}  (287/2pt,153/2pt); 
\draw [-latex,shorten >=5pt] (287/2pt,153/2pt) to node [rectangle,draw,fill=white,sloped,inner sep=1pt] {{\tiny 100011}}  (420/2pt,20/2pt); 
\draw [-latex,shorten >=5pt] (287/2pt,153/2pt) to node [rectangle,draw,fill=white,sloped,inner sep=1pt] {{\tiny 110001}}  (553/2pt,287/2pt); 
\draw [-latex,shorten >=5pt] (87/2pt,353/2pt) to node [rectangle,draw,fill=white,sloped,inner sep=1pt] {{\tiny 110000}}  (287/2pt,553/2pt); 
\draw [-latex,shorten >=5pt] (287/2pt,353/2pt) to node [rectangle,draw,fill=white,sloped,inner sep=1pt] {{\tiny 011000}}  (287/2pt,553/2pt); 
\draw [-latex,shorten >=5pt] (287/2pt,153/2pt) to node [rectangle,draw,fill=white,sloped,inner sep=1pt] {{\tiny 111000}}  (420/2pt,420/2pt); 
\node at (20/2pt,20/2pt) [circle,draw,fill=white,minimum size=10/2pt,inner sep=1pt] {\mbox{\tiny $1$}}; 
\node at (87/2pt,353/2pt) [circle,draw,fill=white,minimum size=10/2pt,inner sep=1pt] {\mbox{\tiny $2$}}; 
\node at (287/2pt,353/2pt) [circle,draw,fill=white,minimum size=10/2pt,inner sep=1pt] {\mbox{\tiny $3$}}; 
\node at (287/2pt,153/2pt) [circle,draw,fill=white,minimum size=10/2pt,inner sep=1pt] {\mbox{\tiny $4$}}; 
\node at (153/2pt,287/2pt) [circle,draw,fill=white,minimum size=10/2pt,inner sep=1pt] {\mbox{\tiny $5$}}; 
\node at (87/2pt,153/2pt) [circle,draw,fill=white,minimum size=10/2pt,inner sep=1pt] {\mbox{\tiny $6$}}; 
\end{scope}
\node at (20/2pt,20/2pt) [circle,draw,fill=white,minimum size=10/2pt,inner sep=1pt] {\mbox{\tiny $1$}}; 
\node at (420/2pt,20/2pt) [circle,draw,fill=white,minimum size=10/2pt,inner sep=1pt] {\mbox{\tiny $1$}}; 
\node at (20/2pt,420/2pt) [circle,draw,fill=white,minimum size=10/2pt,inner sep=1pt] {\mbox{\tiny $1$}}; 
\node at (420/2pt,420/2pt) [circle,draw,fill=white,minimum size=10/2pt,inner sep=1pt] {\mbox{\tiny $1$}}; 
\end{tikzpicture}}
\end{tabular}
\end{center}
\end{example}

\subsection{From dimer to surface}

A dimer on a torus is called \emph{weak Fano}, if it is consistent and its polygon has one internal lattice point. 
This means that the number of elementary triangles in the polygon $\Pg$ equals the number of elementary segments 
on the boundary, so by lemma \ref{pmscor} the last condition is equivalent to the property that the number of zigzag paths 
equals the number of vertices. By chosing the weak paths $x,y$ (used to constuct the lattice points 
for the perfect matchings) carefully we can assume that the internal vertex of the polygon is $(0,0,1)$.
We order the boundary lattice points of the polygon cyclicly and assume they have coordinates $\tilde v_i := (v_i,1)$ where $v_i\in\Z^2$.

To any weak Fano dimer $\qpol$ we will associate the toric Fano surface $X_{\qpol}$ defined by the $v_i$.
This gives us a sequence $(a_i)$ which can be reconstructed from the zigzag cycles. Fix a trivial vertex $\triv$
and use this to assign to each $(v_i,1)$ a stable perfect matching $\PM_i$ and 
to each elementary line segment between $(v_i,1)$ and $(v_{i+1},1)$ a zigzag path $\cZ_i$ (see theorem \ref{pmsmain} and lemma \ref{pmscor}).

\begin{proposition}
The sequence $(a_i)$ of the weak toric Fano is the same as the sequence $(k_i-2)$ where
$k_i$ is the number of common arrows between the $i^{th}$ and $(i+1)^{th}$ zigzag cycle.
\end{proposition}
\begin{proof}
The $i^{th}$ zigzag path that corresponds to the boundary segment $\tilde v_{i-1} \tilde v_{i}$ points in the direction
perpendicular to $v_{i}-v_{i-1}$. Therefore the intersection number on the torus between the $i^{th}$ and $(i+1)^{th}$
zigzag cycle can be calculated as
\se{
 \det \smatrix{v_i - v_{i-1}\\v_{i+1} - v_{i}}
&= 
\det \smatrix{v_i \\v_{i+1} }
+\det \smatrix{- v_{i-1}\\v_{i+1}}
+\det \smatrix{v_i \\- v_{i}}
\det \smatrix{- v_{i-1}\\- v_{i}}\\
&1+\det \smatrix{a_iv_i +v_{i+1}\\v_{i+1}}
+0+1\\
=2 + a_i
}
The intersection number between two consecutive zigzag paths is always equal to the number of joint arrows.
If this were not the case the zigzag paths would have to cross each other in different directions (both from left to right
and from right to left).
This would imply that there are is an arrow $a_1$ for which the first zigzag cycle is a zig ray and the second a zag ray, and
an arrow $a_2$ for which the first zigzag cycle is a zag ray and the second a zig ray.
By lemma \ref{pmscor} (3) one of these arrows  would be contained
in all stable boundary perfect matchings except one (while the other is only contained in a single one). 
Now each arrow is contained in at least one boundary matching 
and each perfect matching contains just one arrows of a positive or negative cycle. 
Therefore $a_1$ or $a_2$ would sit in a positive cycle of at most 2 arrows, which is forbidden by our definition of dimer model.
\end{proof}

Now we construct a full strongly exceptional sequence on this weak Fano surface. 
Fix a trivial vertex $\triv$ and
let $E_i$ be the divisor corresponding to the vector $v_i$ and $\triv$-stable perfect matching $\PM_i$.
Choose a degree function $\cR =\lambda_1\PM_1+\dots +\lambda_n \PM_n$ with $\lambda_i>0$ such that
$(\cR(x),\cR(y),\cR(\ell))=(0,0,2)$ (this is possible because $(0,0,2)$ is in the cone spanned by the polygon).

All cyclic paths in the dimer have a $\cR$-degree that is an even integer so vertex $v$ can be given a number $\cR(p)$ with $h(p)=\triv$ 
and $t(p)=v$. This number is uniquely defined
in $\Rl/2\Z$ and we can use it to give a cyclic order to the vertices of $\qpol$. Let $w_1,\dots w_k$ be these vertices
in cyclic order starting with $w_1=\triv$.

Now put $\cLL_1=\cO_X$ and define $\cLL_{i\pm 1}$ inductively from $\cLL_i$ as follows.
If $p$ is a weak path from $w_i$ to $w_{i\pm 1}$ with $0\le\cR(p)<2$ then we set
\[
 \cLL_{i\pm 1} =\cLL_i\otimes \ccE_1^{\pm\PM_1(p)} \otimes \cdots \otimes \ccE_k^{\pm\PM_k(p)}.
\]

\begin{proposition}
$(\cLL_i)$ is a well-defined cyclic strong exceptional sequence on $X_\qpol$ and its
sequence $(b_i)$ is given by $(\#\{a\in \qpol_1| h(a)=w_{i+1}, t(a)=w_{i}-2\}$ 
\end{proposition}
\begin{proof}
We first have to show that the sequence is well defined.
If we had chosen a different path weak $p'$ then $\cR(p')=\cR(p)$ because
$\cR(c)$ of any cycle is an even integer. Therefore
\se{
\cLL_{i+1}^{-1}\cLL_{i+1}' &= 
\ccE_1^{\PM_1(p')-\PM_1(p)} \otimes \cdots \otimes \ccE_k^{^{\PM_k(p')-\PM_k(p)}} \\
&=\ccE_1^{\<v_1,(i,j)\>} \otimes \cdots \otimes \ccE_k^{\<v_k,(i,j)\>}\cong\cO \text{ in $\Pic X_{\qpol}$}
}

To show that the sequence is strongly exceptional it suffices to show that 
for any weak path $p$ with $0\le \cR(p)<2$ the bundle
\[
 \cLL(p) := \ccE_1^{\PM_1(p)} \otimes \cdots \otimes \ccE_k^{\PM_k(p)}
\]
has $H^i(\cLL(p))=0$ for all $i\ne 0$ and
$H^i(\cLL(p^{-1})=0$ for all $i$.

The cohomology of $\sum_i d_iE_i$ can be computed using the complex $\cF^2\stackrel{\delta}{\to} \cF^1\stackrel{\delta}{\to}\cF^0$
\[
\cF^r := \bigoplus_{|S|=r,S\subset \partial \Pg}\C[X^iY^j| i\PM(x)+j\PM(y)\ge -\PM(p) \forall \PM \in S]
\]
where the sums are over the stable sets of boundary matchings and $$\delta (X^iY^j)_S=\sum_{\PM \in S} \pm (X^iY^j)_{S\setminus\{\PM\}} $$ is the boundary map between these simplicial sets. 
There is homology with $\Z^2$-degree $(i,j)$ 
if and only if the simplicial subcomplex of boundary perfect matchings for which $i\PM(x)+j\PM(y)\ge -\PM(p)$ has homology.

We can rephrase these conditions in terms by using $p'=pX^iY^j$:
The homology $H^u(\cLL(p))\ne 0$ if and only if there is a weak path $p'$ with $h(p)=h(p')$, $t(p)=t(p')$ and
$\cR(p)=\cR(p')$ such that  
\[
\begin{cases}
 \text{$\PM_r(p')\ge 0$ for all $r$} &\text{if $u=0$}\\
 \text{the $r$ for which $\PM_r(p')\ge 0$ do not form a connected segment in $\Z/k\Z$} &\text{if $u=1$}\\
 \text{$\PM_r(p')< 0$ for all $r$} &\text{if $u=1$}
\end{cases}
\]
So $H^2(p')=0$ because $\cR(p')\ge 0$, while $H^1(p')=0$ because of lemma \ref{pmscor} (5).
If $0<\cR(p)<1$ then  $H^0(p^{-1})=0$ because by Poincare duality $H^0(p^{-1})=H^2(p\ell)=0$.

To prove the statement about the sequence $(b_i)$ we need to show that
$\dim\Hom(\cLL_{i},\cLL_{i+1})$ equals the number of arrows between $w_i$ and $w_{i+1}$.
In other words none of the morphism factors. This is indeed true because the $\cR$ for every morphism
is positive but there is no vertex intermediate between $w_i$ and $w_{i+1}$.
\end{proof}
\begin{remark}
If we start wih a weak Fano dimer, construct its surface and full cyclic strongly exceptional sequence
and use this to construct a dimer again we end up with the dimer we started. This follows because 
$L_j:= \Gamma(\pi^*\cL_j)\cong T_{a^p}$  where $a^p_\mu = \cP_\mu(p)$ for $p$ any path from $\triv$ to the vertex 
$w_j$. 
So for every path $p$ starting at $\triv$ we get a module $T_{a^p}$ and $T_{a^p}\subset T_{a^q}$ if $a_{p^{-1}q}\ge 0$.
The latter condition implies that $p^{-1}q$ must be a real path in the dimer. If not let $k>0$ be the minimal power
for which $pq^{-1}\ell^k$ is real. 
From theorem 8.7 in \cite{Bocklandtcons} we know that there is a corner perfect matching $\cP_\mu$ for which 
$\cP_\mu(pq^{-1}\ell^k)=0$ but this would imply thet $a_\mu(pq^{-1}<0$.
We see that there is a bijection between the real paths in the dimer and the embeddings $T_{a^p}\subset T_{a^q}$, and hence
the embeddings that do not factor correspond to the arrows of the original dimer.
\end{remark}

\begin{remark}
Normally one would first construct a tilting bundle on the total space $Y$ using an appropriate stability condition and restrict
this to a strong exceptional sequence on the zero fiber $X$. But we prefered not to do this because in general it is not so straight
forward to cook up the stability condition that does the trick. This can be done using work by Craw and Ishii\cite{IshiiCraw}. 
In general the notion of $\triv$-stability does not give a moduli space of representations that is isomorphic to the total space $Y$. But
the $\triv$-stable perfect matchings are sufficient to generate the strong exceptional sequence without constructing a moduli space of $\theta$-stable representations.
\end{remark}

\begin{theorem}
If $\qpol$ is weak Fano then $\mirror{\qpol}$ is also weak Fano with $(a_i)^{\scriptsize{\mirror{\qpol}}}=(b_i)^{\qpol}$ and
$(b_i)^{\scriptsize{\mirror{\qpol}}}=(a_i)^{\qpol}$.
\end{theorem}
\begin{proof}
Because we already know the number of zigzag paths equals the number of vertices and
dimer duality interchanges both quantities, so both the dimer and its dual are embedded in a torus

Now we prove that $\mirror{\qpol}$ is consistent.
Take a zigzag path $\cZ$ in the original dimer $\qpol$. By lemma \ref{pmscor} (2) we can find $\triv$-stable perfect matchings $\PM_1,\PM_2$ on the boundary of the polygon such that $\cZ=
(\PM_1\cup \PM_2)\setminus (\PM_1\cap \PM_2)$. Let $\cO$ be the left opposite path of $\cZ$. This consists
all rest of the arrows of the positive cycles that meet $\cZ$ except those in $\cZ$ itself. 
This path has degree zero $\PM_1$ and $\PM_2$ and hence it is not a multiple of $\ell$. It can identified with
the monomial $X^{i}Y^{j}Z^k \in Z(\cJA(\qpol))$ with $(-i,-j)$ the homology class of $\cZ$ and $k$ as small as possible.
Because $\Pg$ is a polygon with one internal lattice point, the dual cone must also be generated by
a polygon with one internal lattice point so $k=1$.
 
Now $\cZ=\cO^{-1}\ell^{z/2}$ where $z$ is the length of $\cZ$ and the the zigzag path corresponds to the
element $X^{-i}Y^{-j}Z^{-1+z/2}$. 
Choose a degree function $\cR =\lambda_1\PM_1+\dots +\lambda_n \PM_n$ with $\lambda_i>0$ such that
$(\cR(x),\cR(y),\cR(\ell))=(0,0,2)$ (remember: this is possible because $(0,0,2)$ is in the cone spanned by the polygon). Then $\cR(\cZ)=2+z$ and we can write this as
\[
 \sum_{a\in \cZ}\cR(a)=-2+z \text{ or }\sum_{a\in \cZ}(1-\cR(a))=2
\]
If we look at this condition from the point of view of the dual dimer $\mirror{\qpol}$ this is precisely
the condition for a consistent $\cR$-charge.

The $(a_i)$ and $(b_i)$-sequences are interchanged because the duality interchanges vertices and zigzag paths. 
\end{proof}

\subsection{A categorical poin of view}
Using the dimer duality and the fact that both the dimer and its dual are consistent we get two
equivalences
\[
 \fuk(\qpol)\cong_{\infty}\H\mf(\mirror{\qpol}) \hspace{2cm} \fuk(\mirror{\qpol})\cong_{\infty}\H\mf(\qpol) 
\]
Following \cite{Bocklandtmirror} we can go over to the derived versions of the twisted completions of all these categories. 
In this way we get equivalences between
the derived wrapped Fukaya category of the surface with punctures $|\qpol|\setminus \qpol_0$
on the one hand and the category of singularities of $\cJA(\qpol)/(\ell)$. The latter is by definition
\[
 \D\Sing \cJA(\qpol)/(\ell) := \frac{\D^b\Mod\, \cJA(\qpol)/(\ell)}{\Perf \cJA(\qpol)/(\ell)}.
\]
By a theorem of Ishii and Ueda \cite{IU} the category of singularities
of $\cJA(\qpol)/(\ell)$ is also equivalent to the category of singularities of $f_\qpol^{-1}(0)$, 
Where $f_\qpol:\tilde Y\to \C$ corresponds to the coordinate function $\ell\in \C[Y]=Z(\cJA(\qpol)$.
So
\[
 \D\Fuk(|\qpol|\setminus \qpol_0)\cong \D\Sing f_{\mirror{\scriptstyle{\qpol}}}^{-1}(0)\hspace{2cm}
 \D\Fuk(|\mirror{\qpol}|\setminus \mirror{\qpol}_0)\cong \D\Sing f_{\qpol}^{-1}(0)
\]
Now both $|\qpol|\setminus \qpol_0$ and $|\!\mirror{\qpol}\!|\setminus \mirror{\qpol}_0$ are
tori with the same number of punctures, so we get that the 4 completed categories
above are all equivalent. But, as there is no prescribed isomorphism between the two tori, there seems
to be no canonical isomorphism between $\D\Sing f_{\mirror{\scriptstyle{\qpol}}}^{-1}(0)$ and $\D\Sing f_{\qpol}^{-1}(0)$.

One expects however to be able to identify two objects in $\D\Sing \cJA(\qpol)/(\ell)$: one with
ext-algebra equal to $\H \mf(\qpol)$ (considered as an algebra) and one with ext-algebra $\H \mf(\mirror{\qpol})$.
The former is $\cJA(\qpol)/\bar \cJ=\oplus_{v\in \qpol_0}vS_v$ viewed as a $\cJA(\qpol)/(\ell)$-module because one can easily check that
the resolution of $S_v=v\cJA(\qpol)/\bar \cJ$ stabilizes to 
$$\bigoplus_{a, h(a)=v} \bar P_a\otimes_{\cJA(\qpol)}\cJA(\qpol)/(\ell).$$
We expect the latter to be the direct sums $$\bigoplus_{a, \cZ^+_a=\cZ} \bar P_a\otimes_{\cJA(\qpol)}\cJA(\qpol)/(\ell)$$ 
where $\cZ$ is a zigzag-path. At the moment it
is not clear whether one can find objects $S_\cZ$ in $\Mod\,\cJA(\qpol)/(\ell)$ (or $\Coh f_{\qpol}^{-1}(0)$) that stabilize to these.

\section{An example: dimers for reflexive polygons of size $8$}

We end with an illustration of the main theorem for reflexive polygons of size $8$.
There are $3$ such polygons: a square, a trapezoid and a triangle. The square has $4$ dimers, the trapezoid
$2$ and the triangle $1$. The dimer duality on these $7$ dimers maps fixes 3 of the dimers (2 for the square and one for the trapezoid),
but in these cases no canonical isomorphism between the dimer and its dual. In the diagram below we show all dimers
and we labeled the matching arrows between a dimer and its dual by the same number.  
We drew the polygons next to their dimers in such a way that the zigzag paths point in the same directions
as the normals to the corresponding boundary segments in the polygon.

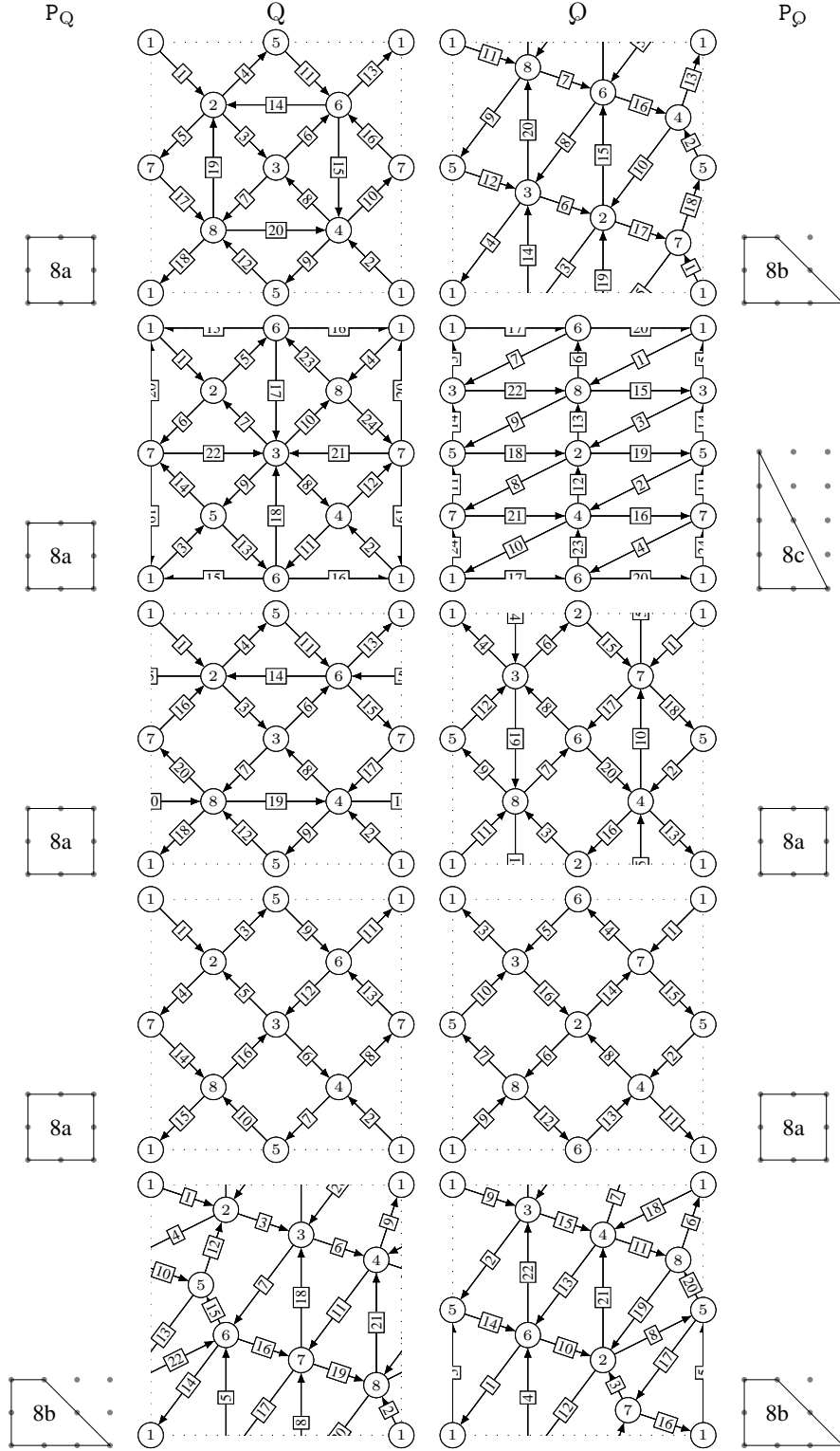
\begin{tabular}{cccc}
$\Pg_{\qpol}$&$\qpol$&$\mirror{\qpol}$&$\Pg_{\mirror{\scriptstyle{\qpol}}}$\\
\end{tabular}

\bibliographystyle{amsplain}

\begin{thebibliography}{10}

\bibitem{abouzaid}
M. Abouzaid, D. Auroux, A. Efimov, L. Katzarkov, D. Orlov
\emph{Homological mirror symmetry for punctured spheres},
{\tt arXiv:1103.4322}

\bibitem{abouzaidwrapped} 
M. Abouzaid, P. Seidel, \emph{An open string analogue of Viterbo functoriality}, Geom. Topol. 14 (2010),
627–718.

\bibitem{Bocklandtqp}
R.~Bocklandt, \emph{Calabi-Yau Algebras and Quiver Polyhedra}, {\tt arXiv:0905.0232}

\bibitem{Bocklandtmirror}
R.~Bocklandt, \emph{Noncommutative mirror symmetry for punctured surfaces}, {\tt arXiv:1111.3392}

\bibitem{Bocklandtnccrs}
R.~Bocklandt, \emph{Generating toric noncommutative crepant resolutions}, {\tt arXiv:1104.1597}

\bibitem{Bocklandtcons}
R.~Bocklandt, \emph{Consistency conditions for dimer models}, {\tt arXiv:1104.1592}

\bibitem{BCQ}
R. Bocklandt, A. Craw, A. Quintero-V´elez, 2012. Work in preparation

\bibitem{bridgeland}
T. Bridgeland. \emph{t-structures on some local Calabi–Yau varieties.} J. Alg., 289(2):453–483, 2005

\bibitem{Broomhead}
N.~Broomhead, \emph{Dimer models and Calabi-Yau algebras}, Memoirs of the AMS, 2011, 86 pp. MEMO/215/1011
{\tt arXiv:0901.4662}

\bibitem{IshiiCraw}
A. Craw, A. Ishii,
\emph{Flops of G-Hilb and equivalences of derived categories by variation of GIT quotient}
Duke Math. J. Volume 124, Number 2 (2004), 259-307.

\bibitem{quivering}
B. Feng, Y. He, K. Kennaway, C. Vafa \emph{Dimer Models from Mirror Symmetry and Quivering Amoebae} Adv. Theor. Math. Phys. Volume 12, Number 3 (2008), 489-545.

\bibitem{perlinghille}
L.~Hille, M.~Perling,
\emph{Exceptional Sequences of Invertible Sheaves on Rational Surfaces}
{\tt arXiv:0810.1936}

\bibitem{IU}
A.~Ishii, K.~Ueda, 
\emph{Dimer models and exceptional collections},
{\tt arXiv:0911.4529}

\bibitem{IUcons}
A.~Ishii, K.~Ueda, 
\emph{A note on consistency conditions on dimer models}
{\tt arXiv:1012.5449}.

\bibitem{Ishii}
A.~Ishii, K.~Ueda, 
\emph{On moduli spaces of quiver representations associated with brane tilings}
Higher dimensional algebraic varieties and vector bundles, 127–141, RIMS Kokyuroku Bessatsu, B9, Res. Inst. Math. Sci. (RIMS), Kyoto, 2008

\bibitem{Keller}
B. Keller, \emph{Introduction to A-infinity algebras and modules}. Homology Homotopy Appl. 3 (2001), 1–35.

\bibitem{Kenyon}
R.  Kenyon,    \emph{An introduction to the dimer model}, School and Conference on Probability Theory, 267–304, ICTP Lect. Notes, XVII, Abdus Salam Int. Cent. Theoret. Phys., Trieste, 2004

\bibitem{Moz} 
S. Mozgovoy,
\emph{Crepant resolutions and brane tilings I: Toric realization}
{\tt arXiv:0908.3475}

\bibitem{MB}
M. Bender, S. Mozgovoy,
\emph{Crepant resolutions and brane tilings II: Tilting bundles}
{\tt arXiv:0909.2013}

\bibitem{VandenBergh}
M.~Van den Bergh, \emph{Non-commutative crepant resolutions}, The Legacy of Niels Hendrik Abel, Springer, pp. 749-770, 2002.

\end{thebibliography}
\def\cprime{$'$}
\providecommand{\bysame}{\leavevmode\hbox to3em{\hrulefill}\thinspace}
\providecommand{\MR}{\relax\ifhmode\unskip\space\fi MR }
\providecommand{\MRhref}[2]{%
  \href{http://www.ams.org/mathscinet-getitem?mr=#1}{#2}
}
\providecommand{\href}[2]{#2}

\end{document}